\documentclass[11pt,draft]{amsart}
\usepackage{amssymb,amsthm,amsmath,latexsym, enumerate, amscd, mathabx,stmaryrd}

\usepackage[all]{xy}

\theoremstyle{plain}
\newtheorem{theorem}{Theorem}[section]
\newtheorem{definition/theorem}[theorem]{Definition/Theorem}
\newtheorem{corollary}[theorem]{Corollary}
\newtheorem{lemma}[theorem]{Lemma}
\newtheorem{proposition}[theorem]{Proposition}

\theoremstyle{definition}
\newtheorem{definition}[theorem]{Definition}
\newtheorem*{convention}{Convention}
\newtheorem{example}[theorem]{Example}

\newtheorem{remark}[theorem]{Remark}
\newtheorem{Definitions and Notation}[theorem]{Definitions and
Notation}

\numberwithin{equation}{section}

\newcommand{\f}[1]{\ensuremath{\mathfrak{#1}}}

\newcommand{\lra}{\longrightarrow}
\newcommand{\sse}{\subseteq}
\newcommand{\ssne}{\subsetneq}
\newcommand{\ol}[1]{\ensuremath{\overline{#1}}}

\newcommand{\depth}{\operatorname{depth}} \newcommand{\Dim}{\operatorname{dim}}
 \newcommand{\Edim}{\operatorname{edim}}

\newcommand{\card}[1]{\ensuremath{\left|#1\right|}}
\newcommand{\ZZ}{\ensuremath{\mathbb{Z}}}
\newcommand{\QQ}{\ensuremath{\mathbb{Q}}}
\newcommand{\zf}[1]{\ensuremath{\operatorname{\mathcal{Z}}\left(#1\right)}}

\begin{document}

\title{Polynomials Inducing the Zero Function on Local Rings}


\author{Mark W. Rogers}
\address{Department of Mathematics, Missouri State University,
Springfield, MO 65897, USA}
\email{markrogers@missouristate.edu}

\author{Cameron Wickham}
\address{Department of Mathematics, Missouri State University,
Springfield, MO 65897, USA}
\email{cwickham@missouristate.edu}




\date{\today}

\begin{abstract}
  For a Noetherian local ring $(R, \f{m})$ having a finite residue field of
  cardinality $q$, we study the connections between the ideal \zf{R} of $R[x]$,
  which is the set of polynomials that vanish on $R$, and the ideal \zf{\f{m}},
  the polynomials that vanish on \f{m}, using what we call
  \emph{$\pi$-polynomials}:  polynomials of the form
  $\pi(x) = \prod_{i = 1}^{q} (x - c_{i})$, where $c_{1}, \ldots, c_{q}$ is a
  set of representatives of the residue classes of \f{m}.  When $R$ is Henselian
  we prove that $\pi(R) = \f{m}$ and show that a generating set for \zf{R} may
  be obtained from a generating set for \zf{\f{m}} by composing with $\pi(x)$.
  When \f{m} is principal and has index of nilpotency $e$, we prove that if
  $e \leq q$ then $\zf{\f{m}} = (x, \f{m})^{e}$, and if $e = q + 1$ then
  $\zf{\f{m}} = (x, \f{m})^{e} + (x^{q} - m^{q - 1}x)$.  When $R$ is finite, we
  prove that $\zf{R} = \bigcap_{i = 1}^{q} \zf{c_{i} + \f{m}}$ is a minimal
  primary decomposition.  We determine when \zf{R} is nonzero, regular, or
  principal, respectively, and do the same for \zf{\f{m}}.  We prove that when
  $R$ is complete, repeated application of $\pi(x) + x$ to elements of $R$ will
  produce a sequence converging to the roots of $\pi(x)$.  We show that \zf{R}
  is the intersection of the principal ideals generated by the
  $\pi$-polynomials.
\end{abstract}

\maketitle

\section{Introduction}

One of the surprising facts about finite rings is that a polynomial can be
nonzero and yet induce the zero function.  An interesting first example is
provided by Fermat's Little Theorem:  If $p$ is prime, then the nonzero
polynomial $x^{p} - x$ induces the zero function on $\ZZ_{p}$.  Using this we
can build obvious examples, such as $p(x^{p} - x)$ and $(x^{p} - x)^{2}$ on
$\ZZ_{p^{2}}$, and more surprising examples, such as $(x^{p} - x)^{p} - p^{p -
  1}(x^{p} - x)$ on $\ZZ_{p^{p + 1}}$; see Corollary~\ref{C:Dickson} for more
information on these examples.  For a ring $R$, it's easy to see that the set of
polynomials in $R[x]$ that induce the zero function on $R$ is an ideal of
$R[x]$; we call this the \emph{zero-function ideal} of $R$, denoted $\zf{R}$.
The zero-function ideal has been studied, often with particular focus on the
rings $R = \ZZ_{p^{n}}$, for its connection with integer-valued polynomials and
functions induced by polynomials \cite{F, J, P, Z} and coding theory \cite{LRS}.

For most of our paper, $(R, \f{m})$ is a Noetherian local ring; we will see that
\zf{R} is only nonzero when the residue field $\ol{R} = R/\f{m}$ is finite, so
we focus most of our attention on this case and let $q = \left|\ol{R}\right|$.
Many authors have studied the problem of finding a generating set for \zf{R},
most often in the case $R = \ZZ_{p^{n}}$ \cite{D, B, F, L, NW, P, W}.  In this
paper, we argue that in many cases, focus should be shifted from \zf{R} to the
simpler ideal \zf{\f{m}}, which is the set of polynomials that induce the zero
function on \f{m}.  The connection between the two is the ideal \zf{R, \f{m}},
the set of polynomials that take elements of $R$ into \f{m}; it's easy to show
that $\zf{R, \f{m}} = (x^{q} - x, \f{m})$. Certainly polynomials in \zf{R,
  \f{m}} can be composed with those in \zf{\f{m}} to obtain polynomials in
\zf{R}; visually, $\zf{\f{m}} \circ \zf{R, \f{m}} \sse \zf{R}$.  One of the main
themes of this paper is to show that in some sense the opposite is true:  A
generating set for \zf{R} can be obtained by composing generators of \zf{R,
  \f{m}} with generators of \zf{\f{m}} (see Theorem~\ref{T:generators}).

\begin{example}\label{E:Z9}
  Let $R = \ZZ_{9}$, so that $\f{m} = (3)$ and $\ol{R} \cong \ZZ_{3}$.  By
  direct computation or by Theorem~\ref{T:Dickson}, $\zf{\f{m}} = (x,
  \f{m})^{2}$.  In Lemma~\ref{L:Stuff} we easily find that with $\pi(x) = x^{3}
  - x$, $\zf{R, \f{m}} = (\pi(x), \f{m}) = (x^{3} - x, 3)$; from this, according
  to Theorem~\ref{T:generators}, we deduce
  \[
  \zf{R} = \zf{\f{m}} \circ \zf{R, \f{m}} = (x^{3} - x, 3)^{2}.
  \]
  This makes it clear that the simpler ideal $\zf{\f{m}}$ controls the structure
  of the generating set of the more complicated ideal $\zf{R}$.
\end{example}

The polynomial $x^{q} - x$ has played an important role in the research on
zero-functions, due primarily to the fact that the image of $x^{q} - x$
generates \zf{\ol{R}} and, to a lesser extent, the fact that when $R$ is finite,
$x^{q} - x$ maps $R$ surjectively onto \f{m}.  See A. Bandini's paper \cite{B}
for applications of surjectivity; we generalize the surjectivity result in
Corollary~\ref{C:Onto} and apply it in Theorem~\ref{T:LikeGilmer} and
Proposition~\ref{P:composition}.  A secondary theme of this paper is that there
is actually a class of polynomials with these properties that can play the role
of $x^{q} - x$; we call these \emph{$\pi$-polynomials}, defined to be
polynomials of the form $\pi(x) = \prod_{i = 1}^{q}(x - c_{i})$ where $c_{1},
\ldots, c_{q}$ is any set of representatives of the residue classes of \f{m}.
As we will see, $x^{q} - x$ is a $\pi$-polynomial when $R$ is Henselian (which
holds true in the common case where $R$ is finite).  If $R$ is complete, we also
provide a computational way of obtaining the factorization of any
$\pi$-polynomial, such as $x^{q} - x$ itself.  In the case of $x^{q} - x$, this
method is as simple as choosing any element of $R$ and repeatedly taking the
$q$th power; the results converge to a root of $x^{q} - x$.  (See
Theorems~\ref{T:EquivalenceOfPi}, \ref{T:Limit}.)

The ideal \zf{R} for $R = \ZZ_{p^{n}}$ was studied as early as 1929 by
L. E. Dickson \cite[Theorem 27]{D}; in that work, the polynomials in \zf{R} were
referred to as \emph{residual polynomials}.  Dickson found a generating set for
\zf{\ZZ_{p^{n}}} when $n \leq p$.  In our notation, he found $\zf{R} = (\pi(x),
\f{m})^{n}$, where $\pi(x) = x^{p} - x$ and $\f{m} = pR$.  We generalize and
recover this work as another application of our Theorem~\ref{T:generators}:  We
show in Theorem~\ref{T:Dickson} and its corollary that if $(R, \f{m})$ is an
Artinian local ring with a principal maximal ideal having index of nilpotency $e
\leq q$, then $\zf{\f{m}} = (x, \f{m})^{e}$, and thus $\zf{R} = (\pi(x),
\f{m})^{e}$ for any $\pi$-polynomial.  When $e > q$, the situation is more
complicated, but we take care of the case $e = q + 1$; the result is related to
results on \zf{R} for specific rings $R$ (\cite[Theorem~2.1]{B} and
\cite[Theorem II]{L}).

As further indication of the importance of $\pi$-polynomials and $\zf{\f{m}}$,
we provide two additional results.  Under suitable conditions, we prove in
Proposition \ref{P:Intersection} that $\zf{R}$ is the intersection of the
principal ideals generated by the $\pi$-polynomials, and in Proposition
\ref{P:Primary} we provide a minimal primary decomposition of \zf{R} as the
intersection of the ideals \zf{c_i + \f{m}}, where $c_{1}, \ldots, c_{q}$ is a
set of represetatives of the residue classes of \f{m}.  Since generators for
\zf{c_i + \f{m}} may be obtained from generators for \zf{\f{m}} by composition
with $x - c_{i}$, this shows that a primary decomposition for \zf{R} may be
obtained from knowing only a generating set for \zf{\f{m}}.  This result on
primary decomposition is a generalization of results from the paper \cite{P} of
G.~Peruginelli, which was concerned with the ring $R = \ZZ_{p^{n}}$.

The remaining theme of our paper is provided in Theorem \ref{T:LikeGilmer},
Theorem~\ref{T:Nonzero}, and Corollary \ref{C:Nonzero}, where we identify
conditions under which \zf{R} is nonzero, principal, and regular, and the same
for \zf{\f{m}}; these results explain why we often focus our attention on finite
rings.  The results generalize, have some overlap with, and were inspired by
R.~Gilmer's paper \cite{G}.

\section{Zero-Function Ideals and $\pi$-polynomials} \label{S:MainResults}

We begin with a precise definition of the zero-function ideal of a ring, and we
define a class of polynomials that plays an important role the study of
zero-function ideals.

\begin{definition}
  Let $R$ be a commutative ring with identity, let $S$ be a subset of $R$, and
  let $J$ be an ideal of $R$.  The set \zf{S,J} of polynomials in $R[x]$ which
  map $S$ into $J$ is easily seen to be an ideal of $R[x]$.  When the ideal $J$
  is omitted, it is assumed to be zero.  The focus of this paper is on \zf{R},
  which we call the \emph{zero-function ideal} of $R$.
\end{definition}

\begin{definition}
  Suppose the local ring $(R, \f{m})$ has a finite residue field.  If $c_{1},
  \ldots, c_{q}$ is any set of representatives of the residue classes of \f{m},
  then we call the polynomial $\pi(x) = \prod_{i = 1}^{q} (x - c_{i})$ a
  \emph{$\pi$-polynomial} for $R$.
\end{definition}





\begin{example}\label{E:Z8}
  We mentioned in the introduction that for $R = \ZZ_{9}$, $\zf{\f{m}} = (x,
  3)^{2}$.  However, for $R = \ZZ_{8}$, $\zf{\f{m}} \supsetneq (x, 2)^{3}$.  In
  fact, according to Theorem~\ref{T:Dickson} and a few brief calculations,
  $\zf{\f{m}} = (x, 2)^{3} + (x^{2} - 2x) = (x^{2} - 2x, 4x)$.  We may then use
  Theorem~\ref{T:generators} to compose with the $\pi$-polynomial $\pi(x) =
  x^{2} - x$ and conclude that $\zf{R} = ((x^{2} - x)^{2} - 2(x^{2} - x),
  4(x^{2} - x))$.
\end{example}



The following basic result is a generalized factor theorem that will be useful
in a few proofs.  The part related to units appears in Gilmer's proof of Theorem
4 in \cite{G}.

\begin{lemma}\label{L:Factor}
  Let $R$ be a commutative ring with identity.  If $f(x) \in R[x]$ is a
  polynomial with roots $c_{1}, c_{2}, \ldots, c_{n}$ such that each difference
  $c_{i} - c_{j}$ ($i \neq j$) is either a regular element or a unit, then
  $(x - c_{1})(x - c_{2}) \ldots (x - c_{n})$ divides $f(x)$ in $R[x]$.
\end{lemma}

\begin{proof}
  By the Factor Theorem, $f(x) = (x - c_{1})f_{1}(x)$ for some $f_{1}(x) \in
  R[x]$.  Substituting $c_{2}$ in for $x$, we obtain $0 = f(c_{2}) = (c_{2} -
  c_{1})f_{1}(c_{2})$.  Since $c_{2} - c_{1}$ is either a regular element or a
  unit, $c_{2}$ is a root of $f_{1}(x)$; we obtain some $f_{2}(x) \in R[x]$ with
  $f_{1}(x) = (x - c_{2})f_{2}(x)$.  So $f(x) = (x - c_{1})(x - c_{2})f_{2}(x)$.
  By continuing to evaluate at the $c$'s in this manner, we arrive at $f(x) = (x
  - c_{1})(x - c_{2}) \cdots (x - c_{n}) f_{n}(x)$ for some polynomial $f_{n}(x)
  \in R[x]$, as desired.
\end{proof}



\begin{convention}
  Throughout this paper, let $R$ be a Noetherian local ring with maximal ideal
  \f{m}.  Unless otherwise specified, the residue field $R/\f{m}$ will be
  denoted by \ol{R}.  The image in \ol{R} of an element $r \in R$ will be
  denoted by \ol{r}.  If the residue field is finite, $c_{1}, \ldots, c_{q}$
  will denote a set of representatives of the residue classes of \f{m} and
  $\pi(x)$ will denote the $\pi$-polynomial $\pi(x) = \prod_{i = 1}^{q} (x -
  c_{i})$.
\end{convention}

The following lemma may be viewed as a generalization of Fermat's Little
Theorem; it is well-known, at least in special cases, as remarked by D. J. Lewis
in \cite{L}.
A simple but important consequence of this lemma is that $\pi(R) \sse \f{m}$.
Later in Corollary~\ref{C:Onto} we will show that if $R$ is Henselian, then
$\pi(R) = \f{m}$.

\begin{lemma}\label{L:Stuff}
  If $(R, \f{m})$ is a Noetherian local ring with finite residue field of
  cardinality $q$, then $\zf{R, \f{m}} = (\pi(x), \f{m})$ for any
  $\pi$-polynomial $\pi(x)$.
\end{lemma}  

\begin{proof}
  Let $f(x) \in \zf{R, \f{m}}$; then $\ol{f}(x) \in \zf{\ol{R}}$.  By
  Lemma~\ref{L:Factor}, $\ol{f}(x)$ is in the ideal generated by $\ol{\pi}(x)$ in
  $\ol{R}[x]$; pull this back to $R[x]$ to get $f(x) \in (\pi(x), \f{m})$.

  For the opposite containment, certainly the constant polynomials in \f{m} are
  in \zf{R, \f{m}}.  Now suppose $\pi(x) = \prod_{i = 1}^{q} (x - c_{i})$.
  Since any element in $R$ is congruent modulo \f{m} to one of the $c_{i}$, the
  polynomial $\pi(x)$ is in \zf{R,\f{m}}, as desired.
\end{proof}

\begin{example}\label{E:NotPi}
  Let $R = \ZZ_{(5)}$ ($\ZZ$ localized at the prime ideal $(5)$), so that $\f{m}
  = (5)$, $\ol{R} = \ZZ_{5}$, and $q = 5$.  The polynomial $x^{q} - x$ is not a
  $\pi$-polynomial since it doesn't factor completely over $R \sse \QQ$:  $x^{5}
  - x = x(x - 1)(x + 1)(x^{2} + 1)$.  However, by Lagrange's Theorem applied to
  the group of units of \ol{R}, it is still true that $x^{5} - x \in \zf{R,
    \f{m}}$.  According to the lemma, we expect $x^{5} - x \in (\pi(x), \f{m})$
  for any $\pi$-polynomial.  In fact, if for example we let $\pi(x) = (x - 2)(x
  + 1)x(x - 1)(x + 2)$, then $x^{5} - x = \pi(x) + 5(x^{3} - 1) \in (\pi(x),
  \f{m})$.
\end{example}

In the next result we show that the zero-functions are precisely those
polynomials that are multiples of each $\pi$-polynomial.

\begin{proposition}\label{P:Intersection}
  Let $(R, \f{m})$ be a Noetherian ring with finite residue field \ol{R} of
  cardinality $q$.  The zero-function ideal of $R$ is the intersection of the
  principal ideals generated by the $\pi$-polynomials:  $\zf{R} = \bigcap
  (\pi(x))$.
\end{proposition}

\begin{proof}
  The fact that any zero-function $f(x) \in \zf{R}$ is a multiple of any
  $\pi$-polynomial $\pi(x)$ follows immediately from Lemma~\ref{L:Factor}.  This
  shows $\zf{R} \sse \bigcap (\pi(x))$.

  Let $f(x) \in \bigcap (\pi(x))$ and let $r \in R$; we show $f(r) = 0$.  We may
  extend $r$ to a set $r, c_{2}, \ldots, c_{q}$ of representatives of the
  residue classes of \f{m}, and thus $\pi(x) = (x - r)(x - c_{2}) \cdots (x -
  c_{q})$ is a $\pi$-polynomial having $r$ as a root.  Since $f(x)$ is a
  multiple of $\pi(x)$, $f(r) = 0$.  Thus $f(x) \in \zf{R}$.


  
\end{proof}

In the following result we give a minimal primary decomposition of \zf{R} if $R$
is finite.  First, we give a simple example:

\begin{example}\label{E:Primary}
  For $R = \ZZ_{9}$, as mentioned in the introduction, we have $\zf{m} = (x,
  3)^{2} = (x^{2}, 3x)$, so the proposition below gives the following minimal
  primary decomposition of \zf{R}:
  \[
  ((x^{3} - x)^{2}, 3(x^{3} - x)) = (x^{2}, 3x) \cap ((x - 1)^{2}, 3(x - 1))
  \cap ((x - 2)^{2}, 3(x - 2)),
  \]
  where the ideals on the right are primary for the maximal ideals $(x, 3)$, $(x
  - 1, 3)$, and $(x - 2, 3)$, respectively.
\end{example}

\begin{proposition}\label{P:Primary}
  Let $(R, \f{m})$ be a finite local ring with residue field $\ol{R}$ of
  cardinality $q$.  Let $c_{1}, \ldots, c_{q}$ be a set of representatives of
  the residue classes of \f{m}.  Then $\zf{R} = \bigcap_{i = 1}^{q} \zf{c_{i} +
    \f{m}}$ is a minimal primary decomposition of \zf{R}.  For each $i$, the
  associated prime of \zf{c_i + \f{m}} is the maximal ideal $(x - c_{i},
  \f{m})$.
\end{proposition}

\begin{proof}
  Certainly any zero-function on $R$ is also zero on each $c_{i} + \f{m}$.  For
  the other inclusion, use the fact that
  $R = \bigcup_{i = 1}^{q} (c_{i} + \f{m})$.  For the statement about associated
  primes, let $e$ be the index of nilpotency of \f{m}, so that $\f{m}^{e} = 0$
  but $\f{m}^{e - 1} \neq 0$, and use the fact that
  $(x - c_{i}, \f{m})^{e} \sse \zf{c_{i} + \f{m}}$.

  As for the minimality of the decomposition, let $j$ be an integer between 1
  and $q$; we show that $\zf{R} \ssne \bigcap_{i \neq j} \zf{c_{i} + \f{m}}$.
  Let $h(x) = \prod_{i \neq j} (x - c_{i})^{e}$; then $h(x) \in \bigcap_{i \neq
    j} \zf{c_{i} + \f{m}}$ since $\f{m}^{e} = 0$.  To see that $h(x)$ is not a
  zero-function, note that $h(c_{j}) = \prod_{i \neq j} (c_{j} - c_{i})^{e}$ is
  a product of units, and is thus nonzero.
\end{proof}

Next we provide an equivalent way to view $\pi$-polynomials, provided the ring
is Henselian; of course, this holds for the finite local rings in which we are
mainly interested.  For an example where the two conditions below are not
equivalent, see Example~\ref{E:NotPi}.  This theorem is also needed in our proof
that $\pi$-polynomials map $R$ surjectively onto \f{m}.

\begin{theorem}\label{T:EquivalenceOfPi}
  Let $R$ be a Henselian local ring with finite residue field \ol{R} of
  cardinality $q$.  For any polynomial $\pi(x) \in R[x]$, the following
  statements are equivalent:
  \begin{enumerate}[(i)]
    \item The polynomial $\pi(x)$ is a $\pi$-polynomial.
    
    \item The polynomial $\pi(x)$ is monic and maps to $x^{q} - x$ in $\ol{R}[x]$.
  \end{enumerate}
\end{theorem}

\begin{proof}
  Let $\pi(x)$ be any $\pi$-polynomial.  Since \ol{R} is a field with $q$
  elements, by Lagrange's theorem on the group of units of \ol{R}, $x^{q} - x$
  is a zero-function on \ol{R}.  By Lemma~\ref{L:Factor}, $\ol{\pi}(x)$ divides
  $x^{q} - x$ in $\ol{R}[x]$.  Since these are monic polynomials of the same
  degree, they are equal.

  For the converse, suppose $\pi(x)$ is any monic polynomial with $\ol{\pi}(x) =
  x^{q} - x$.  Let $\ol{R} = \{\ol{d_{1}}, \ldots, \ol{d_{q}}\}$; as discussed
  in the previous paragraph, $x^{q} - x = \prod_{i = 1}^{q} (x - \ol{d_{i}})$ in
  $\ol{R}[x]$.  By Hensel's Lemma, this factorization of $\ol{\pi}(x)$ can be
  pulled back to a factorization in $R[x]$:  There exist $c_{i}$ in $R$ with
  $\ol{c_{i}} = \ol{d_{i}}$ such that $\pi(x) = \prod_{i = 1}^{q} (x - c_{i})$.
  Thus $\pi(x)$ is a $\pi$-polynomial.
\end{proof}

In the following corollary, we improve upon part of Lemma~\ref{L:Stuff} by
showing that the induced function $\pi \colon R \to \f{m}$ is actually
surjective when $R$ is Henselian.  This generalizes Lemma~1.3 of \cite{B}, where
A.~Bandini proved that, for any prime $p$, $\pi(R) = \f{m}$ in case $R =
\ZZ_{p^{n}}$ and $\pi(x) = x^{p} - x$.  We use this corollary in our
Theorem~\ref{T:LikeGilmer}, where we characterize finite rings with principal
zero-function ideals, expanding upon Gilmer \cite{G}.  It is used again in
Proposition~\ref{P:composition}, which is fundamental for our
Theorem~\ref{T:generators}, which shows that generators for $\zf{R}$ may be
obtained by composing generators for \zf{\f{m}} with a $\pi$-polynomial.

\begin{corollary}\label{C:Onto}
  If $(R, \f{m})$ is a Henselian local ring with finite residue field \ol{R} of
  cardinality $q$, then $\pi(R) = \pi(\f{c}) = \f{m}$ for any $\pi$-polynomial
  $\pi(x)$ and any coset \f{c} of \f{m}.
\end{corollary}

\begin{proof}
  We show that $\pi(\f{c}) \sse \pi(R) \sse \f{m} \sse \pi(\f{c})$.  The first
  containment is clear since $\f{c} \sse R$, and we saw the second containment
  in Lemma~\ref{L:Stuff}.  For the final containment, let $m \in \f{m}$.  By
  Theorem~\ref{T:EquivalenceOfPi}, the polynomial $\pi(x) - m$ is still a
  $\pi$-polynomial, and thus it factors over $R$:  $\pi(x) - m = (x - c_{1})(x -
  c_{2}) \cdots (x - c_{q})$. This shows that for each $i$, $\pi(c_{i}) = m$.
  Since $c_{1}, c_{2}, \ldots, c_{q}$ is a set of representatives of the residue
  classes of \f{m}, one of them, say $c_{j}$, is in \f{c}.  Thus $m = \pi(c_{j})
  \in \pi(\f{c})$, as desired.
\end{proof}




\section{When \zf{R} and \zf{\f{m}} are Nonzero, Regular, or Principal}

In the upcoming Theorems~\ref{T:LikeGilmer} and \ref{T:Nonzero} we will use the
following result from B.~R.~McDonald.  McDonald states and proves the theorem
for any finite local ring $(R, \f{m})$, but the theorem and proof still hold
when $R$ is just Artinian.  The notation McDonald uses is different from ours
but the part we will use is that over an Artinian local ring, any regular
polynomial is an associate of a monic polynomial.  McDonald writes $\mu f$ where
we would write $\ol{f}$, the image of $f$ in $\ol{R}[x]$.

\begin{theorem}\label{T:McDonaldAssociate}\cite[Theorem XIII.6, p. 259]{M}
Let $f$ be a regular polynomial in $R[x]$.  Then there is a monic polynomial
$f^{*}$ with $\mu f = \mu f^{*}$ and, for an element $a$ in $R$, $f(a) = 0$ if
and only if $f^{*}(a) = 0$.  Furthermore, there is a unit $v$ in $R[x]$ with $vf
= f^{*}$.  
\end{theorem}

One of the motivations for the current paper is Theorem 4 from \cite{G}, which
states that if $(R, \f{m})$ is a zero-dimensional local ring, then \zf{R} is
principal if and only if either $\ol{R}$ is infinite (when $\zf{R} = 0$) or $R$
is a finite field (when \zf{R} is generated by $x^{q} - x$.)  The following
result recovers half of Gilmer's result, using some of the same ideas but a few
different ones as well.  For example, Gilmer used a result of E.~Snapper;
instead we use the result of McDonald mentioned above.  Also, we make the
connection with $\pi$-polynomials and \zf{\f{m}}.  The other half of Gilmer's
result is recovered in Theorem~\ref{T:Nonzero}.

\begin{theorem}\label{T:LikeGilmer}
  Let $(R, \f{m})$ be a finite local ring and let $\pi(x)$ be any
  $\pi$-polynomial for $R$.  The following statements are equivalent:
  \begin{enumerate}
  \item $R$ is a field.
  \item \zf{R} is principal.
  \item $\zf{R} = (\pi(x))$.
  \item \zf{\f{m}} is principal.
  \item $\zf{\f{m}} = (x)$.
  \end{enumerate}
\end{theorem}

\begin{proof}
  If $R$ is a field then $R[x]$ is a principal ideal domain, so \zf{R} is
  principal.

  Assume \zf{R} is principal.  Since $\pi(R) \sse \f{m}$ and $\f{m}^{e} = 0$ for
  some $e \geq 1$, \zf{R} contains regular polynomials, such as $\pi(x)^{e}$;
  thus, the generator of \zf{R} must be regular.  According to
  Theorem~\ref{T:McDonaldAssociate}, we may assume the generator is monic:
  $\zf{R} = (f(x))$ for some monic polynomial in $R[x]$.  Let $r$ be a nonzero
  element in the annihilator of \f{m}, so that, by Lemma~\ref{L:Stuff}, $r
  \pi(x)$ is a polynomial of degree $q$ in $\zf{R} = (f(x))$.  This forces
  $f(x)$ to have degree at most $q$, and since we know that all polynomials in
  \zf{R} are multiples of $\pi(x)$ (Proposition~\ref{P:Intersection}), the
  degree of $f(x)$ is exactly $q$.  Since $f(x)$ is a monic multiple of $\pi(x)$
  with the same degree, $f(x) = \pi(x)$.

  If $\zf{R} = (\pi(x))$, then according to Corollary~\ref{C:Onto}, $0 = \pi(R)
  = \f{m}$.  From this we easily conclude that $\zf{\f{m}} = (x)$, and thus
  $\zf{\f{m}}$ is principal.

  If $\zf{\f{m}}$ is principal, we use an argument similar to the part where we
  assumed \zf{R} is principal.  Since \zf{\f{m}} contains $x^{e}$ for some $e
  \geq 1$, \zf{\f{m}} contains regular elements, so the generator of \zf{\f{m}}
  is regular, and we may assume it is monic:  $\zf{\f{m}} = (f(x))$ for some
  monic $f(x) \in R[x]$.  Let $r$ be a nonzero element in the annihilator of
  \f{m}, so that $r x \in \zf{\f{m}}$.  This forces $f(x)$ to
  have degree at most 1.  Since $f(x)$ is monic and $f(0) = 0$, $f(x) = x$, as
  desired.

  If $\zf{\f{m}} = x$, then $\f{m} = 0$, so $R$ is a field.
\end{proof}

In the next proposition it becomes clear how the conditions of being Artinian or
having a finite residue field affect \zf{R} and \zf{\f{m}}.  We will use the
concept of the \emph{embedding dimension} of $R$, denoted $\Edim R$; this is the
minimal number of generators of the maximal ideal.  Recall that $\depth R \leq
\Dim R \leq \Edim R$.

\begin{theorem}\label{T:Nonzero}
  Let $(R, \f{m})$ be a Noetherian local ring.
  \begin{enumerate}
  \item\label{T:NonzeroZfm} $\zf{\f{m}}$ contains nonzero polynomials if and
    only if $\depth R = 0$.
  \item\label{T:NonzeroZfR} $\zf{R}$ contains nonzero polynomials if and only if
    $\depth R = 0$ and \ol{R} is finite.
  \item \zf{\f{m}} contains regular polynomials if and only if $\Dim R = 0$.
  \item \zf{R} contains regular polynomials if and only if $\Dim R = 0$ and
    \ol{R} is finite.
  \item \zf{\f{m}} is generated by a regular polynomial if and only if
    $\Edim R = 0$.
  \item \zf{R} is generated by a regular polynomial if and only if $\Edim R = 0$
    and \ol{R} is finite.
  \end{enumerate}
\end{theorem}

In order to make the similarity with the other parts more clear, parts (4), (5),
and (6) of the previous proposition were not stated as concisely as possible.
Before we present the proof, we state and prove a corollary to clarify those
three parts.  In the development of this paper, we were particularly interested
in the monic polynomials in \zf{R}, so this corollary explains why we were
mainly focused on finite rings.

\begin{corollary}\label{C:Nonzero}
  Let  $(R, \f{m})$ be a Noetherian local ring.
  \begin{enumerate}
    \setcounter{enumi}{2}
  \item \zf{\f{m}} contains regular polynomials if and only if $R$ is Artinian.
  \item \zf{R} contains regular polynomials if and only if $R$ is a
    finite ring.
  \item \zf{\f{m}} is generated by a regular polynomial if and only if $R$ is a
    field.
  \item \zf{R} is generated by a regular polynomial if and only if $R$ is a
    finite field.
  \end{enumerate}
\end{corollary}

\begin{proof}[Proof of Corollary~\ref{C:Nonzero}]
  Part (3) is clear since a Noetherian local ring is Artinian if and only if it
  has dimension 0.  Part (4) follows from the proposition since a zero
  dimensional Noetherian local ring has a finite residue field if and only if
  the ring is finite.  Parts (5) and (6) follow since $\Edim R = 0$ if and only
  if $R$ is a field.
\end{proof}

\begin{proof}[Proof of Theorem~\ref{T:Nonzero}]
  (1): If $R$ has depth 0 then since $R$ is Noetherian, there is a nonzero
  element $m$ that annihilates \f{m}; thus $mx \in \zf{\f{m}}$.

  If $R$ does not have depth 0, then $R$ contains a regular element $t$.  Let
  $g(x) = g_{0} + g_{1} x + \cdots + g_{n} x^{n} \in \zf{\f{m}}$.  Since $g(t) =
  g(t^{2}) = g(t^{3}) = \cdots = g(t^{n + 1}) = 0$, there is a matrix equation
  \[
  \begin{bmatrix}
    1 & t^{1} & t^{2} & \cdots & t^{n}\\
    1 & t^{2} & t^{4} & \cdots & t^{2n}\\
    \vdots & \vdots & \vdots & \ddots & \vdots\\
    1 & t^{n + 1} & t^{(n + 1)2} & \cdots & t^{(n + 1)n}\\
  \end{bmatrix}
  \begin{bmatrix}
    g_{0}\\
    g_{1}\\
    \vdots\\
    g_{n}
  \end{bmatrix}
  =
  \begin{bmatrix}
    0\\
    0\\
    \vdots\\
    0
  \end{bmatrix}
  \]
  The determinant of this Vandermonde matrix is
  \[
  \prod_{1 \leq i < j \leq n + 1} (t^{j} - t^{i}) = \prod_{1 \leq i < j \leq n +
    1} t^{i}(t^{j - i} - 1).
  \]
  Since each $t^{j - i} - 1$ is a unit, this determinant is an associate of a
  power of $t$; namely $t^{k}$ where $k = n(n + 1)(n + 2)/6$.  After multiplying
  both sides of the matrix equation by the adjugate and dividing by the unit, we
  find that $t^{k} g_{i} = 0$ for each $i$.  Since $t^{k}$ is regular, the
  polynomial $g(x)$ is zero, as desired.

  (2):  Assume $\ol{R} = \{\ol{c_{1}}, \ldots, \ol{c_{q}}\}$ and $R$ has depth
  0.  Since $R$ is Noetherian with depth 0, there is a nonzero element $m$ that
  annihilates \f{m}.  By Lemma~\ref{L:Stuff}, the polynomial $\pi(x) = \prod_{i
    = 1}^{q} (x - c_{i})$ is in \zf{R,\f{m}}, so the polynomial $m \pi(x)$ is a
  nonzero element of \zf{R}.

  For the converse, if \ol{R} is not finite, then there is an infinite sequence
  of elements $\{c_{n}\}_{n \geq 1}$ such that no two come from the same residue
  class of \f{m}; thus each difference $c_{j} - c_{i}$ ($j \neq i$) is a unit.
  On the other hand, if $R$ does not have depth 0, then there is a regular
  element $t \in \f{m}$.  Consider the sequence $\{t^{n}\}_{n \geq 1}$; for $i >
  j$ we have $t^{i} - t^{j} = t^{j}(t^{i - j} - 1)$.  Since $t^{j}$ is regular
  and $t^{i - j} - 1$ is a unit, each difference $t^{i} - t^{j}$ ($i \neq j$) is
  regular.

  In either case, we may apply Lemma~\ref{L:Factor} to conclude that any nonzero
  polynomial in the zero-function ideal has arbitrarily high degree.  This is a
  contradiction, so the zero-function ideal contains only the zero polynomial,
  as desired.

  (3):  If $R$ has dimension 0 then $\f{m}^{e} = 0$ for some $e \geq 1$.  Thus
  $x^{e} \in \zf{\f{m}}$.

  Conversely, if $R$ has positive dimension, then for any minimal prime ideal
  \f{p} of $R$, $R/\f{p}$ is an integral domain of positive dimension; in
  particular, it has positive depth.  According to Part (1), \zf{\f{m}/\f{p}} is
  the zero ideal of of $R/\f{p}$.  The image of any $f(x) \in \zf{\f{m}}$ in
  $(R/\f{p})[x]$ is in \zf{\f{m}/\f{p}}, and thus $f(x) \in \f{p}[x]$.  Since
  the coefficients of $f(x)$ are in every minimal prime, they are all nilpotent,
  so $f(x)$ is not regular.

  (4):  Suppose $\ol{R}$ is finite and $\Dim R = 0$.  Let $e$ be such that
  $\f{m}^{e} = 0$ and let $\ol{R} = \{\ol{c_{1}}, \ldots, \ol{c_{q}}\}$.  By
  Lemma~\ref{L:Stuff}, the polynomial $\pi(x) = \prod_{i = 1}^{q} (x - c_{i})$
  is in \zf{R, \f{m}}, so the regular polynomial $\pi(x)^{e}$ is in \zf{R}.

  Conversely, suppose either \ol{R} is infinite or $\Dim R \geq 1$.  If \ol{R}
  is infinite, then \zf{R} does not contain regular polynomials since it is the
  zero ideal, by Part (2).  If $\Dim R \geq 1$, then for every minimal prime
  ideal \f{p} of $R$, $R/\f{p}$ is an integral domain of positive dimension; in
  particular, it has positive depth.  According to Part (2), \zf{R/\f{p}} is the
  zero ideal of of $R/\f{p}$.  The image of any $f(x) \in \zf{R}$ in
  $(R/\f{p})[x]$ is in \zf{R/\f{p}}, and thus $f(x) \in \f{p}[x]$.  This shows
  that $f(x)$ cannot be regular.

  (5):  If $\Edim R = 0$ then $R$ is a field, so $\f{m} = 0$ and $\zf{\f{m}} =
  (x)$.

  Conversely, assume $\zf{\f{m}} = (f(x))$ for some regular polynomial $f(x) \in
  R[x]$.  By Part (3), $\Dim R = 0$, so $\f{m}^{e} = 0$ for some $e \geq 1$.
  Due to this and the result from McDonald (our
  Theorem~\ref{T:McDonaldAssociate}), we may assume $f(x)$ is monic.  By Part
  (1), $\depth R = 0$, so there is an element $m \in \f{m}$ with $mx \in
  \zf{\f{m}}$.  This shows that the monic polynomial $f(x)$ must have degree 1;
  the only choice is $f(x) = x$.  From this we see that $\f{m} = 0$ so $R$ is a
  field.

  (6):  If $\Edim R = 0$ and $\ol{R}$ is finite, then $R$ is a finite field, so
  $\zf{R} = (\pi(x))$ for any $\pi$-polynomial, by Lemma~\ref{L:Stuff}.

  Conversely, assume $\zf{R} = (f(x))$ for some regular polynomial $f(x)$.  From
  Part (4) we know that $\Dim R = 0$ and \ol{R} is finite; this implies that $R$
  is finite.  By Theorem~\ref{T:LikeGilmer}, $R$ is a finite field, so $\Edim R
  = 0$.
\end{proof}

In the following example we illustrate the use of this theorem and contrast the
behavior of the zero-function ideals over rings with finite and infinite residue
fields.

\begin{example}
  The ring $R = \ZZ_{2}\llbracket S,T \rrbracket/(S^{2}, ST)$ is a complete
  Noetherian local ring with depth zero, dimension one, and a finite residue
  field with $q = 2$ elements; let $s$ and $t$ be the images of $S$ and $T$ in
  $R$.  According to Theorem~\ref{T:Nonzero}, \zf{\f{m}} is nonzero but does not
  contain regular polynomials, and the same goes for \zf{R}.  We argue that
  $\zf{\f{m}} = (sx)$ and conclude that $\zf{R} = (s(x^{2} - x))$ by applying
  Theorem~\ref{T:generators}.  Certainly $\zf{\f{m}} \supseteq (sx)$, since the
  annihilator of the maximal ideal of $R$ is $sR$.  For the opposite
  containment, let $\tilde{R} = R/sR \cong \ZZ_{2}\llbracket T \rrbracket$, a
  local Noetherian ring of positive depth.  If $f(x) \in \zf{\f{m}}$, then
  $\tilde{f} \in \zf{\tilde{\f{m}}}$, where $\tilde{\f{m}}$ and $\tilde{f}$ are
  the images of $\f{m}$ and $f$ in $\tilde{R}$ and $\tilde{R}[x]$.  By
  Theorem~\ref{T:Nonzero}~(\ref{T:NonzeroZfm}), $\zf{\tilde{\f{m}}} = 0$, so we
  conclude that $f \in (s)$.  Since $f(0) = 0$, $f \in (s) \cap (x) = (sx)$, as
  desired.

  If we switch to an infinite coefficient ring and residue field, say, $\QQ$
  instead of $\ZZ_{2}$, we still have $\zf{\f{m}} = (sx)$ (nonzero but
  containing no regular polynomials).  However, Theorem~\ref{T:generators} does
  not apply (for one thing, $\pi$-polynomials don't exist).  In fact,
  Theorem~\ref{T:Nonzero}~(\ref{T:NonzeroZfR}) guarantees $\zf{R} = 0$ instead
  of $\zf{R} = (s(x^{2} - x))$.
\end{example}

\section{Obtaining \zf{R} from \zf{m}; Applications}

The following proposition is the key to our main result,
Theorem~\ref{T:generators}.

\begin{proposition}\label{P:composition}
  Let $(R, \f{m})$ be a Henselian local ring with finite residue field \ol{R}
  and let $\pi(x)$ be an arbitrary $\pi$-polynomial.  Any $f(x) \in \zf{R}$ may
  be written in the form
  \[
  f(x) = p_{0}(\pi(x)) + x p_{1}(\pi(x)) + x^{2} p_{2}(\pi(x)) + \cdots + x^{q -
    1} p_{q - 1}(\pi(x))
  \]
  with each $p_{i}(x) \in \zf{\f{m}}$.
\end{proposition}  

\begin{proof}
  Let $f(x) \in \zf{R}$.  In the polynomial ring $R[x,y] = R[y][x]$, $f$ may be
  divided by the monic polynomial $\pi(x) - y$, so that $f(x) = Q(x, y)(\pi(x) -
  y) + R(x,y)$ for some $Q(x, y), R(x, y) \in R[x,y]$ with $R(x,y) = 0$ or the
  degree of $R(x,y)$ with respect to $x$ is less than $q$.  Now set $y = p(x)$
  to obtain $f(x) = p_{0}(\pi(x)) + x p_{1}(\pi(x)) + x^{2} p_{2}(\pi(x)) +
  \cdots + x^{q - 1} p_{q - 1}(\pi(x))$ where the polynomials $p_{i}(y) \in
  R[y]$ are the coefficients of the powers of $x$ in $R(x, y)$.

  It remains to see that each $p_{i}(x) \in \zf{\f{m}}$. Let $m \in \f{m}$.
  Since (according to Corollary~\ref{C:Onto}) $\pi$ maps each coset of \f{m}
  onto \f{m}, there exists a set $c_{1}, c_{2}, \ldots, c_{q}$ of
  representatives of the residue classes of \f{m}, with $\pi(c_i) = m$ for each
  $c_{i}$; each difference $c_{i} - c_{j}$ ($i \neq j$) is a unit.  Since $f(x)
  \in \zf{R}$, we may evaluate $f(x)$ at $c_{i}$ for each $i$ from $1$ to $q$ to
  obtain
  \[
  0 = p_{0}(m) + c_{i} p_{1}(m) + c_{i}^{2} p_{2}(m) + \cdots + c_{i}^{q - 1}
  p_{q - 1}(m).
  \]
  In matrix form, this system becomes
  \[
  \begin{bmatrix}
    0\\
    0\\
    \vdots\\
    0
  \end{bmatrix}
  =
  \begin{bmatrix}
    1 & c_{1} & c_{1}^{2} & \cdots & c_{1}^{q - 1}\\
    1 & c_{2} & c_{2}^{2} & \cdots & c_{2}^{q - 1}\\
    \vdots & \vdots & \vdots & \ddots & \vdots\\
    1 & c_{q} & c_{q}^{2} & \cdots & c_{q}^{q - 1}\\
  \end{bmatrix}
  \begin{bmatrix}
    p_{0}(m)\\
    p_{1}(m)\\
    \vdots\\
    p_{q - 1}(m)
  \end{bmatrix}
  \]
  The matrix is a Vandermonde matrix, and its determinant is $\prod_{1 \leq i <
    j \leq q} (c_{j} - c_{i})$, which is a unit since it is a product of units.
  Thus, the matrix is invertible, so each $p_{i}(m) = 0$, as desired.
\end{proof}  

An application of the previous result, we come to our main theorem, which states
that generators for \zf{R} can be obtained by composing generators for
\zf{\f{m}} with any $\pi$-polynomial, which we could roughly describe by writing
$\zf{R} = \zf{\f{m}} \circ \zf{R, \f{m}}$, if one keeps in mind
Lemma~\ref{L:Stuff} which states that $\zf{R, \f{m}} = (\pi(x), \f{m})$.  An
obvious consequence of the next theorem is that the minimal number of generators
of \zf{R} is less than or equal to the minimal number of generators of
\zf{\f{m}}.

\begin{theorem}\label{T:generators}
  Suppose $(R, \f{m})$ is a Henselian local ring with finite residue field
  $\ol{R}$ of cardinality $q$ and let $\pi(x)$ be an arbitrary $\pi$-polynomial.
  If $\zf{\f{m}} = (F_{1}(x), \ldots, F_{n}(x))$ then $\zf{R} = (F_{1}(\pi(x)),
  \ldots, F_{n}(\pi(x)))$.
\end{theorem}

\begin{proof}
  Since $\pi(R) \sse \f{m}$, certainly $\zf{R} \supseteq (F_{1}(\pi(x)), \ldots,
  F_{n}(\pi(x)))$.  Now let $f(x) \in \zf{R}$.  Use
  Proposition~\ref{P:composition} to write $f(x) = p_{0}(\pi(x)) + x
  p_{1}(\pi(x)) + x^{2} p_{2}(\pi(x)) + \cdots + x^{q - 1} p_{q - 1}(\pi(x))$
  with each $p_{i}(x) \in \zf{\f{m}}$.  Since each $p_{i}(x)$ is an
  $R[x]$-linear combination of $F_{1}(x), \ldots, F_{n}(x)$, each
  $p_{i}(\pi(x))$ is an $R[x]$-linear (actually $R[\pi(x)]$-linear) combination
  of $F_{1}(\pi(x)), \ldots, F_{n}(\pi(x))$.  Since $f(x)$ is an $R[x]$-linear
  combination of the $p_{i}(\pi(x))$, the proof is complete.
\end{proof}

\begin{remark}
  The equality $\zf{R} = \zf{\f{m}} \circ \zf{R, \f{m}}$ should not be taken too
  literally.  Certainly polynomials in \zf{\f{m}} composed with polynomials in
  \zf{R, \f{m}} are in \zf{R}, but it's not true that every polynomial in \zf{R}
  can be obtained in that way.  For example, $x(x^{2} - x) \in \zf{\ZZ_{2}}$,
  but since it's degree is not even, it does not equal $f(x^{2} - x)$ for any
  polynomial $f(x)$.
\end{remark}

As mentioned in the introduction, the theorem below is a version of results of
Dickson \cite[p.~22, Theorem 27]{D}, Bandini \cite[Theorem~2.1]{B}, and Lewis
\cite[Theorem II]{L}, adapted for \zf{\f{m}} rather than \zf{R}, and for finite
local rings rather than specific rings.  We then recover the results for \zf{R}
in Corollary~\ref{C:Dickson} as an application of our main theorem,
Theorem~\ref{T:generators}.

\begin{theorem}\label{T:Dickson}
  Let $(R, \f{m})$ be a finite local ring with principal maximal ideal
  $\f{m} = (m)$; set $q = \card{R/\f{m}}$.  Suppose $e$ is the index of
  nilpotency of \f{m}.  If $e \leq q$ then $\zf{\f{m}} = (x, m)^{e}$; if
  $e = q + 1$, then $\zf{\f{m}} = (x, m)^{e} + (x^{q} - m^{q - 1}x)$.
\end{theorem}

\begin{proof}
  We prove the first result using induction on $e$.  The base case $e = 1$ is
  clear, since then $R$ is a field, $\f{m} = 0$, and $\zf{\f{m}} = (x)$.  Assume
  the result is true for rings whose maximal ideal has index of nilpotency $e -
  1 \leq q$; we prove the result for a ring whose maximal ideal has index of
  nilpotency $e \leq q$.  The containment $\supseteq$ is clear.  Let $f(x) \in
  \zf{\f{m}}$; then $\ol{f}(x) \in \zf{R/\f{m}^{e - 1}}$.  By induction,
  $\ol{f}(x) \in \ol{(x, m)^{e - 1}}$, and thus $f(x) \in (x, m)^{e - 1}$.  We
  have $f(x) = \sum_{k = 0}^{e - 1} x^{k} m^{e - 1 - k} f_{k}(x)$ for some
  $f_{k}(x) \in R[x]$; it remains to see that each $f_{k}(x) \in (x, m)$,
  i.e. that $f_{k}(0) \in \f{m}$.

  For each $r \in R$,
  \[
  0 = f(rm) = \sum_{k = 0}^{e - 1} (rm)^{k} m^{e - 1 - k} f_{k}(rm) = m^{e - 1}
  \sum_{k = 0}^{e - 1} r^{k} f_{k}(rm).
  \]
  Since the annihilator of $m^{e - 1}$ is \f{m}, $\sum_{k = 0}^{e - 1} r^{k}
  f_{k}(rm) \in \f{m}$, and thus $\sum_{k = 0}^{e - 1} r^{k} f_{k}(0) \in
  \f{m}$.  This shows that $\sum_{k = 0}^{e - 1} \ol{f_{k}(0)} x^{k} \in
  \zf{\ol{R}}$.  Since by Lemma~\ref{L:Stuff} $\zf{\ol{R}} = (x^{q} - x)$, this
  polynomial with degree less than $q$ must be the zero polynomial. Therefore
  each $f_{k}(0) \in \f{m}$, as desired.

  For the second result, the only part of the containment $\zf{\f{m}} \supseteq
  (x, m)^{e} + (x^{q} - m^{q - 1}x)$ that is not clear is $x^{q} - m^{q - 1}x
  \in \zf{\f{m}}$; for this, take any $rm \in \f{m}$ and compute $(rm)^{q} -
  m^{q - 1}(rm) = m^{q}(r^{q} - r) \in \f{m}^{q + 1} = 0$.  For the opposite
  containment, assume $f(x) \in \zf{\f{m}}$ and reduce module $\f{m}^{e - 1}$ as
  above to obtain a similar expression for $f(x)$, and again deduce that
  $\sum_{k = 0}^{e - 1} \ol{f_{k}}(0) x^{k} \in \zf{\ol{R}} = (x^{q} - x)$.
  Since $e - 1 = q$, we must have $\sum_{k = 0}^{e - 1} \ol{f_{k}}(0) x^{k} =
  \ol{u}(x^{q} - x)$ for some unit $u \in R$; thus each $\ol{f_{k}(0)} = 0$
  except for $\ol{f_{q}(0)} = \ol{u}$ and $\ol{f_{1}(0)} = -\ol{u}$ .  Define
  polynomials $g_{k}(x)$ identical to $f_{k}(x)$ except for $g_{q}(x) = f_{q}(x)
  - u$ and $g_{1}(x) = f_{1}(x) + u$.  Now the constant term of each $g_{k}(x)$
  is in \f{m} and we have
  \[
  f(x) = \sum_{k = 0}^{e - 1} x^{k} m^{e - 1 - k} f_{k}(x) = u(x^{q} - m^{q -
    1}x) + \sum_{k = 0}^{e - 1} x^{k} m^{e - 1 - k} g_{k}(x)
  \]
  so that $f(x) \in (x, m)^{e} + (x^{q} - m^{q - 1}x)$, as desired.
\end{proof}

The following corollary follows immediately from the theorem and
Theorem~\ref{T:generators}.

\begin{corollary}\label{C:Dickson}
  Let $(R, \f{m})$ be a finite local ring with principal maximal ideal $\f{m} =
  (m)$; set $q = \card{R/\f{m}}$.  Suppose $e$ is the index of nilpotency of
  \f{m}, and let $\pi(x)$ be any $\pi$-polynomial.  If $e \leq q$ then $\zf{R} =
  (\pi(x), m)^{e}$; if $e = q + 1$ then $\zf{R} = (\pi(x), m)^{e} + (\pi(x)^{q}
  - m^{q - 1}\pi(x))$.
\end{corollary}

\section{Factoring $\pi$-Polynomials}


The following lemma is the heart of a more constructive approach
(Theorem~\ref{T:Limit}) to the converse part of the proof of
Theorem~\ref{T:EquivalenceOfPi}, which gave two equivalent conditions for
$\pi$-polynomials.  Note that according to this lemma,
if $\f{m}^{e} = 0$, then for any $r \in R$, the sequence $\{p_{n}(r)\}$
stabilizes at $n = e - 1$.

\begin{lemma}\label{L:Cauchy}
  Let $(R, \f{m})$ be a Noetherian local ring with finite residue field \ol{R}
  of cardinality $q$, and let $\pi(x)$ be any polynomial mapping to $x^{q} - x$
  in $\ol{R}[x]$.  Let $p_{0}(x) = x$ and $p_{n}(x) = \pi(p_{n - 1}(x)) + p_{n -
    1}(x)$, so that $p_{n}(x)$ denotes the function obtained by successively
  applying the function $\pi(x) + x$, $n$ times.  For every $n \geq 1$,
  \[
  p_{n}(x) - p_{n - 1}(x) \in \zf{R, \f{m}^{n}}.
  \]
\end{lemma}

\begin{proof}
  We use induction on $n$.  For the base case ($n = 1$) just note that $p_{1}(x)
  - p_{0}(x) = (\pi(x) + x) - x = \pi(x) \in \zf{R, \f{m}^{1}}$ by
  Lemma~\ref{L:Stuff}.
  
  Now assume the induction hypothesis:  For some $n \geq 1$, $p_{n}(x) - p_{n -
    1}(x) \in \zf{R, \f{m}^{n}}$.  We show that $p_{n + 1}(x) - p_{n}(x) \in
  \zf{R, \f{m}^{n + 1}}$.  Since $\pi(x)$ is a polynomial mapping to $x^{q} - x$
  modulo $\f{m}[x]$, there is some $m(x) \in \f{m}[x]$ such that $\pi(x) = x^{q}
  - x - m(x)$; thus $p_{n}(x)$ may also be viewed as applying $x^{q} - m(x)$,
  $n$ times.  For simplicity of notation, set $c = p_{n}(x)$, $a = p_{n -
    1}(x)$, and $b = c - a$, so that $c = \pi(a) + a = a^{q} - m(a)$.  We have
  \[
  \begin{split}
    p_{n + 1}(x) - p_{n}(x) &= \pi(p_{n}(x)) + p_{n}(x) - p_{n}(x)\\
    &= \pi(c)\\
    &= c^{q} - c - m(c)\\
    &= (a + b)^{q} - c - m(c)\\
    &= a^{q} + q a^{q - 1} b + \binom{q}{2} a^{q - 2} b^{2} + \cdots + b^{q} - c
    - m(c)\\
    &= q a^{q - 1} b + \binom{q}{2} a^{q - 2} b^{2} + \cdots + b^{q} + m(a) -
    m(c)
  \end{split}
  \]
  since $a^{q} - c = m(a)$.  By induction, $b \in \zf{R, \f{m}^{n}}$, and since
  $q \in \f{m}$, we see that the first term is in \zf{R, \f{m}^{n + 1}}.  Since
  $b \in \zf{R, \f{m}^{n}}$ and $n \geq 1$, $b^{2} \in \zf{R, \f{m}^{n + 1}}$,
  which takes care of all but the last two terms:  $m(a) - m(c)$.
  
  Now $m(a) - m(c) = \sum_{i = 1}^{\deg m(x)} m_{i} (a^{i} - c^{i})$, where
  $m_{i}$ is the coefficient of $x^{i}$ in $m(x)$, and is thus in \f{m}.  Since
  $-b = a - c$ is a factor of $a^{i}- c^{i}$ for all positive integers $i$ and
  $-b \in \zf{R, \f{m}^{n}}$, it follows that $m(a) - m(c) \in \zf{R, \f{m}^{n +
      1}}$, as desired.
\end{proof}

The following theorem provides, in particular, a more constructive approach to
the proof of the result in Theorem~\ref{T:EquivalenceOfPi} which states that any
monic polynomial $\pi(x)$ mapping to $x^{q} - x$ is actually a $\pi$-polynomial.
In a ring with $\f{m}^{e} = 0$, it allows discovery of the roots by successively
applying the function $\pi(x) + x$ ($e - 1$ times) to representatives of the
residue classes of \f{m}.  When $\pi(x) = x^{q} - x$, this amounts to
successively taking $q$th powers.  In the case of a finite ring, the resulting
roots of $x^{q} - x$ are called \emph{Teichm{\"u}ller elements} in Jian Jun
Jiang's paper \cite{J}.

\begin{theorem}\label{T:Limit}
  Let $(R, \f{m})$ be a complete Noetherian local ring with finite residue field
  $\ol{R} = \{\ol{c_{1}}, \ldots, \ol{c_{q}}\}$.  Let $\pi(x)$ be any polynomial
  mapping to $x^{q} - x$ in $\ol{R}[x]$ and let $p_{n}(x)$ be the function
  obtained by applying $\pi(x) + x$ successively, $n$ times.  The limit
  $\lim_{n \to \infty} p_{n}(c_{i})$ exists.  Set
  $d_{i} = \lim_{n \to \infty} p_{n}(c_{i})$; then $d_{i}$ is a root of $\pi(x)$
  and $\ol{d_{i}} = \ol{c_{i}}$.  If $\pi(x)$ is monic, then there is a
  factorization
  \[
  \pi(x) = (x - d_{1})(x - d_{2}) \cdots (x - d_{q}),
  \]
  and thus $\pi(x)$ is a $\pi$-polynomial.
\end{theorem}

\begin{proof}
  The limit exists since, by Lemma~\ref{L:Cauchy}, the sequence
  $\{p_{n}(c_{i})\}_{n \geq 1}$ is a Cauchy sequence.  For any $r \in R$,
  $\pi(r) + r$ and $r$ are in the same coset of \f{m}, since $\pi(r) =r^{q} - r
  -m(r) \in \f{m}$.  We may apply this fact successively, beginning with $r =
  c_{i}$, to see that each $p_{n}(c_{i})$ is congruent to $c_{i}$ modulo \f{m}.
  Since \f{m} is closed under the \f{m}-adic topology, we conclude that
  $\ol{d_{i}} = \ol{c_{i}}$.

  To see that $\pi(d_{i}) = 0$, use the Cauchy sequence mentioned above and the
  fact that polynomials are continuous under the \f{m}-adic topology:
  \[
  \pi(d_{i}) = \pi(\lim p_{n}(c_{i})) = \lim \pi(p_{n}(c_{i})) = \lim (p_{n +
    1}(c_{i}) - p_{n}(c_{i})) = 0.
  \]

  If $\pi(x)$ is monic then it must have degree $q$; an application of
  Lemma~\ref{L:Factor} completes the proof.
\end{proof}

\begin{example}
  With $R = \ZZ_{125}$, we have $q = 5$ and $e = 3$.  We can choose elements $0,
  1, 2, 3, 4$ to be representatives of the elements of $R/\f{m} = \ZZ_{5}$.  We
  factor the polynomial
  \[
  \pi(x) = x^5 + 5x^4 + 40x^3 + 85x^2 + 24x + 50 = x^5 - x - m(x)
  \]
  where $m(x) = -(5x^{4} + 40x^{3} + 85x^{2} + 25x + 50)$.  Applying $p_{2}(x)$
  to $0, 1, 2, 3, 4$ yields $50, 31, 72, 18, 74$.  According to the theorem,
  $\pi(x)$ factors in $R[x]$ as
  \[
  \pi(x) = (x - 50)(x - 31)(x - 72)(x - 18)(x- 74).
  \]
\end{example}

\bigskip

\noindent {\bf Acknowledgment:} The authors thank Richard Belshoff, Jung-Chen
Liu and Les Reid for helpful discussions and suggestions.

\end{document}